\documentclass[english]{amsart}
\usepackage{amsmath,amssymb,amsfonts,amscd,a4}
\usepackage[cmtip, all]{xy}
\usepackage{pstricks}
\usepackage{color}
\usepackage{ulem}  
\usepackage{verbatim} 
\usepackage{url}

\theoremstyle{plain}
\usepackage{amsthm}
\newtheorem{ithm}{Theorem}[section]
\newtheorem{thm}{Theorem}[section]

\newtheorem{lemma}[thm]{Lemma}
\newtheorem{corollary}[thm]{Corollary}
\newtheorem{prop}[thm]{Proposition}
 \theoremstyle{remark}
\newtheorem{para}[thm]{\bf} 

\theoremstyle{definition}
\newtheorem{remark}[thm]{Remark}

\newtheorem{definition}[thm]{Definition}

\numberwithin{equation}{section}

\newcommand{\ga}[2]{\begin{gather}\label{#1}#2 \end{gather}}

\newcommand{\surj}{\twoheadrightarrow}

\newcommand{\Hom}{{\rm Hom}}
\newcommand{\Ext}{{\rm Ext}}

\newcommand{\dR}{{\rm dR}}
\newcommand{\Spec}{{\rm Spec \,}}

\newcommand{\Char}{{\rm char}}

\newcommand{\Gal}{{\rm Gal}}
\newcommand{\SIsom}{\underline{\rm Isom}}
\newcommand{\SAut}{\underline{\rm Aut}}

\newcommand{\rk}{{\rm rk}}

\newcommand{\sD}{{\mathcal D}}

\newcommand{\sH}{{\mathcal H}}

\newcommand{\sN}{{\mathcal N}}
\newcommand{\sO}{{\mathcal O}}

\newcommand{\sS}{{\mathcal S}}

\newcommand{\C}{{\mathbb C}}

\newcommand{\F}{{\mathbb F}}

\newcommand{\cM}{{\mathcal M}}

\newcommand{\Q}{{\mathbb Q}}

\newcommand{\Z}{{\mathbb Z}}

\newcommand{\iso}{\overset\sim\rightarrow}

\title {$D$-modules and finite monodromy} 
\author{H\'el\`ene Esnault and Mark Kisin} 
\address{Freie Universit\"at Berlin, Arnimallee 3, 14195 Berlin,  Germany}
\email{esnault@math.fu-berlin.de}
\address{Department of Mathematics, Harvard University, 1 Oxford Str., Cambridge, MA 02138, USA  
 }
\email{ kisin@math.harvard.edu}
\thanks{The first  author is supported by  the Einstein program. The second author was partially supported by NSF grant DMS-0017749000}

\date{\today} 

\begin{document}

\begin{abstract}  
We investigate an analogue of the Grothendieck $p$-curvature conjecture, where the vanishing of the $p$-curvature is replaced by the stronger condition, that the 
module with connection mod $p$ underlies a $\sD_X$-module structure. We show that this weaker conjecture holds in various situations, for example if the 
underlying vector bundle is finite in the sense of Nori, or if the connection underlies a $\Z$-variation of Hodge structure. We also show isotriviality assuming 
a coprimality condition on certain mod $p$ Tannakian fundamental groups, which in particular resolves in the projective case  a conjecture of Matzat-van der Put.  

\end{abstract}
\maketitle
 
\section{Introduction}\label{intro}

Let $X/\C$ be a smooth complex  variety, and $M = (E,\nabla)$  a vector bundle with an integrable connection on $X.$  Recall that the Grothendieck $p$-curvature conjecture \cite[I]{Kat72} 
predicts that $M$ has finite monodromy if it has a full set of algebraic solutions mod $p$ for almost all primes $p.$ 
More precisely, we can descend $(X,M)$ to a finitely generated $\Z$-algebra $R \subset \C$ and consider its reductions $(X_s, M_s)$ at closed points $s \in \Spec R.$ 
We consider the following condition:
\begin{quote}
 $\frak{P}:$ there is a dense open subscheme $U\hookrightarrow {\rm Spec} \ R$ such that for all closed points $s\in U$, $M_s$ has $p$-curvature $0$. 
\end{quote}
The conjecture says that this condition implies the existence of a finite \'etale cover $h: Y\to X$ such that $h^*M$ is trivial as a connection; that is $M$ is isotrivial.  
The conjecture is known to be true when the monodromy representation of $M$ is solvable  
(\cite[Thm.~8.5]{Chu85},   \cite[Thm.~2.9]{Bos01}, \cite[Cor.~4.3.2]{And04}),)
 and 
for Gau{\ss}-Manin connections \cite[Thm.~5.1]{ Kat72}. 

The condition $\frak{P}$ means that, Zariski locally, $M_s = (E_s,\nabla_s)$ is spanned by the kernel of $\nabla_s.$ This is equivalent to asking that the action of derivations on $E_s,$ 
given by $\nabla_s$,  extends to an action of differential operators of the form $\frac{(\partial /\partial x)^p} p$ where $x$ is a local co-ordinate on $X.$ Of course, if $M$ becomes trivial over a finite cover then this condition holds, but in that case one has, in fact, a stronger condition: the action of the derivations extends to an action of the full ring of differential operators $\sD_{X_s}.$ 
With this motivation, we consider in this paper the stronger condition
\begin{quote}
$\frak{D} :$ there is a dense open subscheme $U\hookrightarrow {\rm Spec} \ R $ such that for all closed points $s\in U$, $M_s$ 
underlies a $\sD_{X_s}$-module. 
\end{quote}

We denote by $MIC^{\frak{D}}(X/\C)$ the category of vector bundles with integrable connections on $X$ which satisfy $\frak{D}.$

 Unfortunately, we were not able to show that $\frak{D}$ implies that $M$ is isotrivial.
\footnote{In fact  one can show that $\frak{D}$ implies the isotriviality of $M$ by combining Theorem~\ref{ithm:finitenessats} below with Theorem~\cite[Thm.~3.3]{Mat06}.
Unfortunately,  there is a mistake in the proof of Theorem~\cite[Thm.~3.1]{Mat06}, so the proof of Theorem~\cite[Thm.~3.3]{Mat06} is incomplete.}
However, we show some partial results. We assume for the rest of the introduction that $X$ is projective. By \cite[Thm.8.1 (1), (3)]{Kat82}, the $p$-curvature conjecture
can be reduced to this case. See also \S \ref{ss:MvdP} for a discussion on the projectivity assumption. 
Our first theorem shows that one can deduce the isotriviality if we add a finiteness condition on the underlying vector bundle $E.$

\begin{ithm}\label{ithm:Nori} The forgetful functor $(E,\nabla) \mapsto E$ from $MIC^{\frak{D}}(X/\C)$ to the category of vector bundles on $X,$ 
is fully faithful.  In particular, if $E$ is Nori finite, then $M$ is isotrivial.
\end{ithm}
Recall that Nori finiteness means that the class of $E$ in the Grothendieck group associated with the monoid of vector bundles on $X$ (see \cite[Section~2.3]{Nor82})
is integral over $\Z$, or equivalently, as we are in characteristic $0$,  that there is a finite \'etale cover $h: Y\to X$, such that $f^*E$ is trivial as a vector bundle; that is   $E$ is isotrivial. 
Our next result is an analogue of Katz's theorem on the Gau\ss-Manin connection. 

\begin{ithm}\label{ithm:katz} If $M$ in $MIC^{\frak{D}}(X/\C)$ underlies a polarizable $\Z$-variation of Hodge structure, then $M$ is isotrivial.
\end{ithm}

To prove these results we use arguments involving stability of vector bundles, together with the following theorem, set purely in characteristic $p.$
\begin{ithm}\label{ithm:finitenessats} Let $X_0$ be a smooth projective, geometrically connected scheme over a finite field $k,$ and let $M_0$ be a coherent $\sD_{X_0}$-module on $X_0.$ 
Then $M_0$ is isotrivial.
\end{ithm}

This is proved using the existence of the course moduli space of stable vector bundles, and a finiteness argument. 
A consequence of this last theorem, is that any $M_0$ as in the theorem defines a finite \'etale Tannakian group scheme. 
Returning to $M$ over $X/\C$ satisfying $\frak{D}$, we define for $s$ in a non-trivial open in $\Spec R$ the corresponding 
\'etale group scheme $G_s = G(M_s).$ We denote by $k(s)$ the residue field of $s.$ Our final result is 

\begin{ithm}\label{ithm:finiteness} If there is a Zariski dense set of closed points $s \in \Spec R$ such that 
 $\Char \, k(s)$ does not divide the order of $G_s,$ then $M \in MIC^{\frak{D}}(X/\C)$ is isotrivial.
\end{ithm}

In particular this proves a projective analogue of a conjecture of Matzat-van der Put \cite[p.~51]{MP03}, which predicted isotriviality assuming the order of the $G_s$ was bounded independently of $s.$
In fact we explain in \S \ref{ss:MvdP}  
that the original conjecture in {\it loc.~cit}, which dealt with a Zariski open in the affine line, is not correct. 
 
The paper is organized as follows. In \S \ref{sec:p}, we prove Theorem~\ref{ithm:finitenessats}. The method has already been used in \cite{BK08}  (see Remark~\ref{rmk:BK}) and \cite{EM10}. We push it further to prove isotriviality of the whole $\sO$-coherent $\sD$-module. From it we deduce in \S \ref{sec:0} that the underlying vector bundle $E$ of $M$ satisfying $\frak{D}$ is semistable, and is  stable if $M$ is irreducible (Proposition~\ref{prop:ss}).  We then deduce Theorem \ref{ithm:Nori} and use Hitchin-Simpson theory, to deduce Theorem~\ref{ithm:katz}. To prove  isotriviality in Theorem~\ref{ithm:finiteness}, one uses an idea of Andr\'e, who applied  Jordan's theorem \cite{Jor78} in \cite[7.1.3.~Cor.]{And04} to reduce the $p$-curvature conjecture to the case of a number field. This idea in equal characteristic $p>0$ was carried over in \cite[Thm.~5.1]{EL13}, where the coprimality to $p$ appeared as a necessary condition (\cite[Section~4]{EL13}).  The difficulty of the mixed characteristic version presented here is made easier by the fact that the  groups in characteristic $p>0$ are finite. 

{\it Acknowledgements:} We thank Yves Andr\'e, Antoine Chambert-Loir
and Johan de Jong for their interest and for discussions. We especially thank Sinan \"Unver for a close, and perceptive reading of the manuscript, and Jo\~ao Pedro dos Santos for mentioning \cite{Mat06} to us.  The first named author thanks the  department of  mathematics of Harvard University for hospitality during the preparation of this work. 

\section{   $\sO$-coherent $\sD$-modules over finite fields.    } \label{sec:p}
\begin{para} 
Let $X$ be a smooth, geometrically connected, projective variety  over a field $k.$ We fix an ample line bundle $\sO_X(1)$ on $X.$ 
For a coherent sheaf $E$ on $X,$ we set $p_E(n) = \chi(E(n))/\rk\, E,$ where $\chi$ denotes the Euler characteristic. 
We say that $E$ is $\chi$-semi-stable (resp.~$\chi$-stable) if for all proper subsheaves  $E' \subset E$ one has 
$p_{E'}(n) \leq p_E(n)$ (resp. $p_{E'}(n) < p_E(n)$) for $n$ sufficiently large (see \cite[\S 0]{Gie77}, \cite[p.~512]{Lan14}). 
Similarly if $(E,\nabla)$ is a vector bundle with integrable connection, we defined $\chi$-(semi-)stability in the same 
way, but we require $E'$ to be $\nabla$-stable.
\end{para}

\begin{para}
Suppose that $k = \F_{p^a}$ is a finite field. 
For $i \geq 0,$ we denote by $X^{(i)}$ the pullback of $X$ by $F_k^i,$ where $F_k$ is the absolute Frobenius on $k.$ 
We have the relative Frobenius maps $F: X^{(i)}\to X^{(i+1)}.$ 

Let $M$ be a $\sO_X$-coherent $\sD_X$-module.   Associated to $M$ we have a Frobenius divided sheaf 
$(E^{(i)}, \sigma^{(i)})_{i\ge 0}$, where $E^{(i)}$ is a vector bundle on $X^{(i)},$ with $E^{(0)} = M,$ and $\sigma^{(i)}$ is an isomorphism 
 $E^{(i)} \iso F^*E^{(i+1)}$ over $X^{(i)}$. In fact the categories of $\sO_X$-coherent $\sD_X$-modules and of Frobenius divided sheaves are equivalent  (\cite[Thm.~1.3]{Gie75}). 
The $E^{(i)}$ have trivial numerical Chern classes as these classes are infinitely $p$-power divisible and by definition lie in $\Z$.  

Then $M$ generates a $k$-linear  Tannakian subcategory $\langle M\rangle $ of the category of $\sO_X$-coherent $\sD_X$-modules.
If $x\in X(k),$  taking the fibre at $x$ of $E^{(0)}$  defines a  neutral fibre functor on $\langle M\rangle.$ 
Let $G(M, x)$ be the Tannaka group of $\langle M\rangle$. 
\end{para}

\begin{thm} \label{thm:fin_tan}
The group scheme $G(M, x)$ over $k$ is finite \'etale.
\end{thm}
\begin{proof}
Suppose first that the $E^{(i)}$ are $\chi$-stable for $i\geq 0.$
Let $\cM$ be the coarse moduli space of $\chi$-stable vector bundles with  numerical vanishing Chern classes and rank equal to  $\rk \, M,$ which exists over $k$  (\cite[Thm.~1.1]{Lan14}). 
Then $\cM$ is quasi-projective  (\cite[Thm.~0.2]{Lan04}), and in particular has finitely many $k$-points. This implies that 
there exist $i \geq 0$ and $t> 0$ such that $E^{(ai)}$ and $E^{(ai+at)}$ correspond to the same point $[E^{(ai)}]=[E^{(ai+at)}]$ in $\cM(k).$ 
Hence $F^{a*}$ induces a well defined, surjective map on $S = \{[E^{(ai)}]; i \geq 0\} \subset \cM(k).$ 
That is, $F^{a*}$ is an automorphism of the finite set $S.$ It follows that 
there is a natural number $t>0$ such that the points $[E^{(ait)}]$ for $i\ge 0$ are all equal. 
This means that the vector bundles $E^{(0)}$ and $E^{(ait)}$ are isomorphic over an algebraic closure $\bar k$ of $k.$
It follows from Lemma \ref{lem:isomvect} below that they are isomorphic over $k.$

We have an isomorphism $F^{*at}E^{(i)} \iso E^{(i)}$ for $i \geq t$ divisible by $a,$ and hence for $i \geq 0.$
Therefore, there is a finite \'etale cover $h: Y\to X$ such that $h^* E^{(i)}$ is a trivial algebraic bundle for all $i\ge 0$ (\cite[Satz~1.4]{LS77}),   
and it follows that the $\sO_Y$-coherent $\sD_Y$-module $h^*M$ is trivial (\cite[Prop.~1.7]{Gie75}). 
Consequently, in the category of $\sO_X$-coherent $\sD_X$-modules, $\langle M\rangle \subset \langle h_*\sO_Y\rangle$, and $G(\langle h_*\sO_Y\rangle,x)$ is finite \'etale.
Thus $G(\langle M\rangle,x)$ is finite \'etale.  

Now consider the case of arbitrary $M.$  
By \cite[Prop.~2.3]{EM10},   there is a natural number $i_0$ such that $ (E^{(i)}, \sigma^{(i)})_{i\ge i_0a}$ is a successive extension of Frobenius divided sheaves 
$U_n$ on $X^{(i_0a)}$ all of whose underlying vector bundles $U_n^{(i)}$ are stable with vanishing numerical Chern classes.  
It suffices to prove the theorem with $(E^{(i)}, \sigma^{(i)})_{i\ge 0}$ replaced by $(E^{(i+i_0)}, \sigma^{(i+i_0)})_{i\ge 0},$ so we may assume $i_0 =0.$

By what we have seen above there exists a finite \'etale cover  $h: Y \to X$ such that $h^* (\oplus_n U_n)$ is a trivial $\sO_Y$-coherent $\sD_Y$-module. Then $h^*M$ is a successive extension of the trivial $\sO_Y$-coherent $\sD_Y$-module by itself. By induction on the number of factors $U_n,$ we may assume that $h^*M$ is an extension of $(\sO_Y)^{s_1}$ by $(\sO_Y)^{s_2}$ as $\sD_Y$-modules for some $s_1,s_2 >0.$ 
Thus, $h^*M$ is given by a matrix of classes in $\Ext^1_{\sD_Y}(\sO_Y,\sO_Y).$
Arguing with $F$-divided sheaves as above, but replacing the finiteness of $\cM(k)$ by the finiteness of the set 
$\Ext^1_{\sO_Y}(\sO_Y,\sO_Y),$ one finds that any class in $\Ext^1_{\sD_Y}(\sO_Y,\sO_Y)$ 
becomes trivial over a finite \'etale cover of $X.$ Thus $\langle M\rangle$ is finite, and $G(M, x)$ is \'etale as $M$ has  a finite \'etale trivializing cover. 
\end{proof}

\begin{lemma}\label{lem:isomvect} Let $V_1,V_2$ be vector bundles on $X,$ which are isomorphic over $\bar k.$ 
Then $V_1,V_2$ are isomorphic over $k.$
\end{lemma}
\begin{proof} This is presumably well known. Consider the $k$-scheme $\SIsom(V_1,V_2),$ which assigns to any $k$-algebra $R,$ the set of invertible elements in 
$\Hom(V_1,V_2)\otimes_kR.$ Since $V_1,V_2$ are isomorphic over $\bar k,$ this is a torsor under the $k$-group scheme $\SAut(V_1),$ 
whose $R$ points are given by the units in $\Hom(V_1,V_1)\otimes_kR.$ 

For $y \in \SIsom(V_1,V_2)(\bar k)$ and $\sigma \in \Gal(\bar k/k),$ write $\sigma(y) = y\circ c_{\sigma},$ 
with $c_{\sigma} \in \SAut(V_1)(\bar k).$ Then 
 $ (c_{\sigma})$ is a cocycle, defining a class in  $ H^1(\Gal(\bar k/k), \SAut(V_1)(\bar k)).$ Since $\SAut(V_1)$ is Zariski open in a $k$-vector space, it is smooth and connected, 
so this class is trivial by Lang's lemma (\cite[Thm.~2]{Lan56}).  Thus the cocycle  $(c_\sigma)$ is a coboundary, which means that  it is
 the translate of the given point by some element of  $\SAut(V_1)(\bar k)$ is a $k$-point.
\end{proof}

\begin{para} Recall, \cite[Section~3]{Nor82} that a vector bundle $V$ on $X$ is called Nori-finite 
if its class  in the Grothendieck group associated with the monoid of vector bundles on $X$ (see \cite[Section~2.3]{Nor82})
is integral over $\Z$. Equivalently,  there is a torsor under a finite group scheme  $h: Y\to X \otimes k'$, such that $f^*E$ is trivial as a vector bundle. Here $k'\supset k$ is a  finite field extension such that $X(k') \neq \emptyset$.  
 Nori-finite bundles are, in particular, strongly semistable (that is, the bundle and all its Frobenius pullbacks are semi-stable) vector bundles with vanishing numerical Chern classes \cite[Cor.~3.5]{Nor82}.
The category $\sN(X)$ of Nori-finite bundles is Tannakian. For any $x \in X(k),$ taking the fibre at $x$ is a  neutral fibre functor on $\sN(X),$ and 
each object $E$ has a finite Tannakian group scheme $G(E,x)$. 
\end{para}
\begin{corollary} \label{cor:stab}
\begin{itemize}
\item[(1)]
The vector bundles  $E^{(i)}$ are Nori-finite. In particular, they are strongly $\chi$-semistable with vanishing numerical Chern classes. 
\item[(2)]  If $E^{(i)}$ is $\chi$-stable for some  $i\ge 0,$ then $E^{(i)}$ is $\chi$-stable for all $i \geq 0.$
\item[(3)] If $M$ is $\chi$-stable as a module with integrable connection, then $M = E^{(0)}$ is $\chi$-stable as a vector bundle.
\end{itemize}
\end{corollary}
\begin{proof}
(1) is an immediate consequence of Theorem~\ref{thm:fin_tan}, as the statement may be checked over the finite \'etale cover on which the $E^{(i)}$ become trivial.

To see (2), one may use the periodicity of the sequence $\{E^{(i)}\},$ which we saw in the proof of Theorem \ref{thm:fin_tan},
together with the fact that $E^{(i)}$ $\chi$-stable implies $E^{(i+1)}$ $\chi$-stable.

Finally for (3), $\chi$-stability of $M$ as a module with integrable connection, is equivalent to $\chi$-stability of $E^{(1)}$ as a vector bundle. So (3) follows from (2).
\end{proof}

\begin{remark} \label{rmk:BK} 
Corollary~\ref{cor:stab} (1)  and the part concerning the isotriviality  of the bundle $E^{(0)}$ in Theorem~\ref{thm:fin_tan} 
are proven in \cite[Prop.~2.5]{BK08}. There only boundedness is used, not the existence of a coarse moduli space defined over the finite field. 
The latter argument seems essential here, and 
the stronger statement in Theorem~\ref{thm:fin_tan} is used in 
Corollary~\ref{cor:stab}  to conclude stability of $E^{(0)}$ in (2) and (3).
\end{remark}

\begin{para} We still assume $k=\F_{p^a}$. 
Corollary~\ref{cor:stab} enables one to define the forgetful  functor 
\ga{1}{
\frak{forg}: 
\sD(X/k)
 \to  \sN(X); \ 
M\mapsto E^{(0)}, \notag}
which is a tensor functor compatible with the Tannakian structures on both sides.  Here $\sD(X/k)$ is 
the category of  $\sO_X$-coherent  $\sD_X$-modules.   For $x \in X(k),$ we denote by 
\ga{2}{ \frak{forg}^*: \pi_1(\sN(X),x)\to \pi_1(\sD(X/k),x), \ \  \frak{forg}|_M^*:  G(E^{(0)},x)\to G(M,x) \notag}
the induced homomorphisms of Tannaka group schemes. 
\end{para}

\begin{thm} \label{thm:surj_p} The functor $\frak{forg}$ is fully faithful, and for $M$ in $\sD(X/k)$ it induces an equivalence 
$\langle M \rangle \iso \langle E^{(0)} \rangle.$ 

In particular, for $x \in X(k),$ the homomorphism $\frak{forg}^*$ is faithfully flat, and for any $M$ in $\sD(X/k),$ the homomorphism $\frak{forg}|_M^*$ is an isomorphism.
\end{thm}
\begin{proof}
The full faithfulness of $\frak{forg}$ is equivalent to the  surjectivity  of the  induced $k$-linear map $\Hom_{\sD_X}(\sO_X,M)\hookrightarrow H^0(X, E^{(0)}).$
That is we have to show any global section of $H^0(X, E^{(0)})$ gives rise to a map $\sO_X\hookrightarrow E^{(0)}$  of $\sD_X$-modules. This can be checked \'etale locally, so the statement follows from 
Theorem \ref{thm:fin_tan}. To show that $\langle M \rangle \iso \langle E^{(0)} \rangle,$ it suffices to show that for any $M$ in $\sD(X/k),$ a subbundle 
$E' \hookrightarrow  E^{(0)}$ in $\sN(X)$  
is a $\sD_X$-submodule. Since this statement is \'etale local, we may assume that $M,$ and hence also $E',$ is trivial, when the result follows from the full faithfulness 
proved above.

It follows that $\frak{forg}|_M^*$ is an isomorphism, and the faithful flatness of $\frak{forg}^*$ follows from \cite[Prop.~2.21]{DM82}.
\end{proof}

\section{Integrable connections in characteristic $0$ which satisfy $\frak{D}$. } \label{sec:0}
\begin{para} In this section, we derive the consequences in characteristic $0$ of the previous section.  Let $X$ be a smooth,  geometrically connected  scheme of finite type  defined over a field $k$ of characteristic $0,$ and equipped with an ample line bundle $\sO_X(1)$. 

The category of $\sO_X$-coherent $\sD_{X/k}$-modules,  is equivalent to the category of vector bundles with integrable connections $MIC(X/k)$, 
which is a $k$-linear Tannakian category, neutralized by taking the fibre of the underlying vector bundle at any point $x \in X(k)$ (if one exists).
\end{para}

\begin{definition} \label{defn:model}
Let $M=(E,\nabla)\in MIC(X/k)$.  Let $R \hookrightarrow k$ be a ring of finite type over $\Z$.  A model $(X_R, \sO_{X_R}(1), M_R)$ of $(X, \sO_X(1), M)$ over $R$ 
is a smooth, projective $R$-scheme $X_R$ with geometrically connected fibres, equipped with an ample line bundle $\sO_{X_R}(1),$ together with a 
vector bundle with an integrable connection $M_R$ relative to $R,$ and an isomorphism of $(X_R,\sO_{X_R}(1), M_R)\otimes_Rk$ with $(X,\sO_X(1),M).$
\end{definition}
Models always exist over some finitely generated $\Z$-algebra $R$ (see \cite[IV, \S 8]{EGA}).
We fix a model $(X_R, \sO_{X_R}(1), M_R)$ of $(X, \sO_X(1), M).$ 
For $x \in X(k),$ we denote by $G(M,x)$ the Tannaka group of $\langle M\rangle$, the full subcategory of $MIC(X/k)$ spanned by $M.$

\begin{para} \label{defn:pd}
Recall the conditions $\frak{P}$ and $\frak{D}$ from the introduction. 
We define the full Tannakian subcategories 
$$MIC^{\frak{f}}(X/k)\subset MIC^{\frak{D}}(X/k)\subset MIC^{\frak{P}}(X/k)\subset MIC(X/k)$$
 of objects 
 which are finite for $^{\frak{f}}$ (that is they become trivial over a finite \'etale cover of $X$), 
 which verify $\frak{D}$ for $^{\frak{D}}$,
 which verify $\frak{P}$ for $^{\frak{P}}$.
Clearly, the conditions $\frak{D}, \ \frak{P}$  do not depend on the $R$ chosen in Definition~\ref{defn:model}.  All these categories are  Tannakian subcategories of $MIC(X/k)$. 
Grothendieck's $p$-curvature conjecture predicts that 
$$MIC^{\frak{f}}(X/k) = MIC^{\frak{D}}(X/k) = MIC^{\frak{P}}(X/k) \subset MIC(X/k).$$

For the remainder of this subsection we assume that $X$ is projective.
\end{para}

\begin{prop} \label{prop:ss}
If $M=(E, \nabla)\in MIC^{\frak{D}}(X/k)$, then $E$ is $\chi$-semistable with vanishing numerical Chern classes. If $M$ is irreducible, then $E$ is $\chi$-stable.
\end{prop}
\begin{proof}
A destabilizing subsheaf of $E$ would destabilize $E_s=E_s^{(0)}$ for all closed points of some non-empty open in $\Spec R,$  which would contradict Corollary~\ref{cor:stab} (1).  This proves the first statement. As for the second one,  by definition, $M$ is irreducible if and only if it is $\chi$-stable in $MIC(X/k).$ By openness of stability, $M_s$ is $\chi$-stable for all closed points of some non-empty open in $\Spec R$  (\cite[Thm.~1.1]{Lan14}). Thus $E_s^{(0)}$ is $\chi$-stable by Corollary~\ref{cor:stab} (3), and so $E$ is stable.
\end{proof}

\begin{para} 
Let $\sS(X)$ denote the category of semistable vector bundles $E$ on $X$, with vanishing numerical Chern classes. This is a Tannakian category, and taking the fibre at $x \in X(k)$ yields a fibre functor.
Proposition~\ref{prop:ss} enables us to define the forgetful functor 

\ga{1}{
{\rm forg}: 
MIC^{\frak{D}}(X/k)
 \to \sS(X), \ 
M= (E,\nabla)\mapsto E , \notag}
which is a tensor functor compatible with the Tannakian structures on both sides.  
For $x \in X(k),$ we denote by 
\ga{2}{{ \rm forg}^*: \pi_1(\sS(X),x)\to \pi_1(MIC(X/k),x), \ \  {\rm forg}|_M^*:  G(E,x)\to G(M,x) \notag}
the induced homomorphisms of Tannaka group schemes. 
\end{para}
\begin{thm} \label{thm:surj_0}
The functor ${\rm forg}$ is fully faithful. 
The homomorphism ${\rm  forg}|_M^*$ is an isomorphism and  the homomorphism ${\rm forg}^*$ is faithfully flat.  
\end{thm}
\begin{proof}
We argue as in the proof of Theorem~\ref{thm:surj_p}. 
The full faithfulness of ${\rm  forg}$ is equivalent to the surjectivity of the map $H^0_{dR}(X, M)\hookrightarrow H^0(X, E)$ induced by ${\rm forg},$ 
which follows 
from the full faithfulness in Theorem~\ref{thm:surj_p}, by taking the fibres of sections at closed points $s \in \Spec R.$

Next let $M=(E,\nabla)$ be in $MIC^{\frak{D}}(X/k)$ and $E' \hookrightarrow E$  any  subvector bundle in $\sS(X).$ 
Then for $s$ in a non-empty open in $ {\rm Spec} \ R,$ $E'_s$ is semistable with vanishing numerical Chern classes and $E_s$ is trivialized by a finite \'etale cover. 
Thus  $E'_s$ is  trivialized by a finite \'etale cover
 as well, and so lies in $\sN(X_s).$ It follows by Theorem~\ref{thm:surj_p}, that $E'_s \hookrightarrow E_s$  is $\nabla$-stable, 
and hence $E'$ is $\nabla$-stable.

Finally,  this  implies that ${\rm  forg}|_M^*$ 
is an isomorphism and ${\rm forg}^*$ is faithfully flat (\cite[Prop.~2.21]{DM82}).   
\end{proof}

\begin{corollary} If $M = (E,\nabla)$ is in $MIC^{\frak{D}}(X/k)$ and $E$ is in $\sN(X),$ then $M\in MIC^{\frak{f}}(X/k)$.
\end{corollary}
\begin{proof} In this case $G(E,x)$ is a finite (\'etale) group scheme, so the corollary follows from the fact that ${\rm forg}|_M^*$ is an isomorphism.
\end{proof}

\begin{thm} \label{thm:katz}
Let  $X$ be a smooth projective connected variety over $\C$, and $M$ a polarizable $\Z$-variation of Hodge structure, such that $M\in MIC^{\frak{D}}(X/\C)$.  Then $M\in MIC^{\frak{f}}(X/\C)$.
\end{thm}

\begin{proof} 
By the Lefschetz hyperplane theorem, and Bertini's theorem, we can choose $x \in X(\C)$ so that there exists a smooth projective curve $C \subset X,$ with $x \in C,$ and 
such that the map $\pi_1(C,x) \rightarrow \pi_1(X, x)$ is surjective. One checks immediately that 
$M|_C$ is in $MIC^{\frak{D}}(C/\C).$ Hence we may replace $X$ by $C,$ and assume that that $X$ has dimension $1$. 

Deligne's semi-simplicity theorem \cite[4.2]{Del71} over $\Q$, together with the fact that a  summand of a $\Q$-variation of Hodge structure definable over $\Z$ 
is itself definable over $\Z$,  implies that we may assume that $M$ is irreducible, that is stable. It follows by Proposition~\ref{prop:ss} that $E$ is stable.

We apply Hitchin-Simpson theory. 
The semistable Higgs bundle  $(V, \theta)$ associated to $M=(E,\nabla)$  is  $V=gr^F E=\oplus_{a=0}^n \sH^{n-a,a}$ with $\theta: \sH^{n-a,a}\to \omega_X\otimes \sH^{n-a-1, a+1}$ the Kodaira-Spencer map of $\nabla$ (\cite[Thm.~8]{Sim90}). Here $\omega_X$ is the sheaf of differential $1$-forms on $X$.
Choose $a$ as large as possible such that $\sH^{n-a,a} \neq 0.$ 
Then $(\sH^{n-a,a},0)\subset (V, \theta)$ is a Higgs subbundle  and therefore ${\rm deg} \ \sH^{n-a,a} \le 0$. 
On the other hand, by definition, one has the surjection $E\surj \sH^{n-a,a}.$ 
Since $E$ is stable, ${\rm deg}(E)=0$, and $\sH^{n-a,a} \neq 0,$ it follows that 
$E=\sH^{n-a,a}.$ In particular, $\theta \equiv 0,$ and we may apply Katz's argument  \cite[Prop.~4.2.1.3]{Kat72} to conclude that the monodromy of 
$(M,\nabla)$ is finite.  
\end{proof}

\section{Integrable connections in characteristic $0$ which satisfy $({\frak{D}}, p)$.}

\begin{para} We keep the assumptions of the previous section, so in particular $X$ is smooth, projective and geometrically connected over $k.$ 
We  assume that $X(k)$ is non-empty and we fix a point $x \in X(k).$ 
After increasing $R,$ we may assume that $x$ arises from a point $x_R \in X_R(R).$ For a point $s \in \Spec R,$ we denote by $x_s$ 
the image of $x_R$ in $X_R(k(s)).$

If $M\in MIC^{\frak{D}}(X/k),$ then for all closed points $s$ of some non-trivial open in $\Spec R$,  the restriction $M_s$ of a model has a finite \'etale Tannaka group $G_s : = G(M_s, x_s)$ (see Theorem~\ref{thm:fin_tan}),  which does not depend on the choice of $\sD_{X_s}$-module structure on $M_s,$ by Theorem \ref{thm:surj_p}.
We denote by  $|G_s|$ the order of the group scheme $G_s.$ That is, $|G_s|$ is the order of $G_{\bar s} : = G_s(\overline{k(s)}),$ where $\overline{ k(s)}$ is an algebraic closure of the residue field $k(s).$
The order $|G_s|$ does not depend on the rational point chosen, as by Tannaka theory  the isomorphism class of $G_{\bar s}$ does not depend on the choice of the  fibre functor. 
\end{para}

\begin{para}
The group $G_{\bar s}$ may be viewed as a quotient of the geometric \'etale fundamental group $\pi_1(X_{\bar s},x_{\bar s}):$ 
Let $\sO_{G_s}$ denote the Hopf algebra of $G_s.$ By \cite[\S 2]{Nor76}, the $k(s)$-representation $\sO_{G_s}$ corresponds via Tannaka duality to a $G_s$-torsor $P.$ 
Since the only $G_s$-invariant elements of $\sO_{G_s}$ are the constants, we have $H^0(P,\sO_P) = k(s),$ and so $H^0(P_{\overline{ k(s)}}, \sO_{P_{\overline{ k(s)}}}) = \overline{k(s)}.$ 
Hence $P$ is geometrically connected. In particular the automorphism group of $P_{\overline{ k(s)}}/X_{\overline{ k(s)}}$ must be equal to $G_{\bar s},$ and we obtain a surjective 
map $\pi_1(X_{\bar s},x_{\bar s}) \rightarrow G_{\bar s}.$
\end{para}

\begin{definition}
\label{defn:p} Let $MIC^{\frak{D},p}(X/k)\subset MIC^{\frak{D}}(X/k)$ denote the full subcategory of objects $M$ 
such that $|G_s|$ is prime to the characteristic of $k(s)$ for a dense set of closed points of some non-trivial open subset of $\Spec R.$ 
This category does not depend on the choice of model $(X_R, \sO_{X_R}(1), M_R).$ 
One has  inclusions of Tannakian categories 
$$MIC^{\frak{f}}(X/k)\subset MIC^{\frak{D},p}(X/k)
\subset MIC^{\frak{D}}(X/k)\subset MIC^{\frak{P}}(X/k)\subset MIC(X/k).$$
\end{definition}
We finish the paper with a proof of the following 
\begin{thm} \label{thm:Dp}
Let $X$ be a smooth projective geometrically connected   variety over a field $k$ of characteristic $0$, with a rational point $x$. Then 
$$ MIC^{\frak{f}}(X/k)=MIC^{\frak{D},p}(X/k).$$
Moreover, if $G = G(M,x)$ denotes the (finite) monodromy group of $M,$ then we have $G_{\bar s} = G$ for all closed points  $s$ is a non-empty Zariski open subset of $\Spec R.$
\end{thm}

\begin{proof} Take  $M\in MIC^{\frak{D},p}(X/k).$ Consider the dense  set of  closed points $s$ in $\Spec R$ 
such that $G_s$ is defined and $G_{\bar s}$ has order prime to the characteristic of $k(s)$. We denote by $x_{\bar k} \in X(\bar k)$ and $x_{\bar s} \in X(\overline{ k(s)})$ the geometric points 
induced by $x$ and $x_s$ respectively.

By Jordan's theorem \cite{Jor78},  there is a constant $c(r)$ depending only on $r = \rk \, M$ such that 
$G_{\bar s}$ contains a normal abelian subgroup $N_s$  of index at most $c(r)$. 
Thus,  the  surjective specialization homomorphism $\pi_1(X_{\bar k}, x_{\bar k})\surj \pi_1(X_{\bar s}, x_{\bar s})$,  composed with $\pi_1(X_{ \bar s}, x_{\bar s})
 \to G_{\bar s}/N_s$ defines, for each closed point $s,$ a finite quotient of 
 $\pi_1(X_{\bar k}, x_{\bar k})$  of order bounded above by $c(r)$. Since $\pi_1(X_{\bar k}, x_{\bar k})$   is topologically finitely generated, there are finitely many such quotients, and all such 
maps factor through some finite quotient of $\pi_1(X_{\bar k}, x_{\bar k}),$ 
which defines a Galois cover $h: Y\to X_{\bar k}.$ The map $h$ is defined over a finite extension $K$ of $k$, say $h_K: Y_K\to X_K,$ and we may assume that $x$ is the image of a point of $Y_K(K).$
Replacing $k$ by $K$, $X$ by $Y_K,$ 
and $M$ by its pullback to $Y_K,$   the new Tannaka groups $G_s$ are subgroups of the old ones, so that $M\in MIC^{\frak{D},p}(X/k).$
Thus we may assume that $G_s$ is abelian for $s$ in a dense set of closed points in ${\rm Spec} \ R$. 

If $M' = (E',\nabla)$ is an irreducible subquotient of $M$ in $MIC(X/k),$ then $M'$ is stable, and so $M'_s$ is stable for $s$ any closed point  in a non-empty open in $\Spec R,$ and in particular for all $s$ in a dense set of closed points of $\Spec R$ on which $G_s$ is abelian.
This implies that $M'_s$ has rank $1,$ and hence so does $M'$. 
 It follows that $M$ is a successive extension of rank $1$ objects in $MIC(X/k).$
We now apply Andr\'e's solution to Grothendieck's conjecture, for connections with solvable monodromy \cite[Cor.~4.3.2]{And04} to conclude that $M$ has finite monodromy.

We have $G_{\bar s} \subset G,$ and it remains to show that $G = G_{\bar s}$ for all $s$ in a non-empty Zariski open subset of $\Spec R.$
If not, there is a proper subgroup $H \subset G$ such that $G_{\bar s}$ is identified with $H$ for $s$ in a Zariski dense set $T.$ 
By the Tannakian formalism, there is an object $V$ in $\langle M \rangle$ corresponding to a non-trivial, irreducible representation $\rho$ of $G,$ such that $\rho$ admits non-trivial $H$-invariants. 
The latter condition implies that for $s$ in a Zariski dense subset of $T,$ $H^0_{\dR}(X_s, V_s) \neq 0,$ which implies that $H^0_{\dR}(X,V) \neq 0,$ either by
 base change for de Rham cohomology \cite[Thm.~8.0]{Kat70} or Theorem \ref{thm:surj_0}. This contradicts the irreducibility of $\rho.$
\end{proof}

\begin{para}  \label{ss:MvdP} Theorem \ref{thm:Dp} answers an analogue of a question of Matzat - van der Put  \cite[p.~51]{MP03} in the projective case. 
More precisely, in our terminology their question amounts to whether the following assertion holds: Let $k$ be a number field, $X \subset \mathbb A^1_k$ a Zariski open subset, and $M$ in $MIC^{\frak{D}}(X/k).$ Suppose that for almost all $s \in \Spec \sO_k,$ $M_s$ underlies a $\sD_{X_s}$-module which becomes trivial over a finite \'etale Galois covering with group $G_{\bar s} = G,$ a fixed group independent of $s.$ Then $M$ has monodromy group $G.$ When $X$ is replaced by a projective $k$-scheme, this is a particular case of Theorem~\ref{thm:Dp}. 

We remark that if $X$ is {\it not} projective, then the assertion of \cite[p.~51]{MP03} does not hold. Indeed, suppose that $X \hookrightarrow \mathbb A^1_k$ is open with $k$ a number field, and let 
$M = (E,\nabla)$ be a regular connection in $MIC(X/k)$ having finite, non-trivial monodromy. 
Then $M$ has vanishing $p$-curvatures, and so $E_s$ descends to a vector bundle $E_s^{(1)}$ on $X_s^{(1)}.$ 
As $X_s^{(1)}$ is open in  $\mathbb A^1_{k(s)},$ $E_s^{(1)}$ is necessarily a trivial bundle, and so $M_s$ is trivial as an object of $MIC(X_s/k(s)).$ In particular $M_s$ is obtained from a trivial 
$\sD_{X_s}$-module, and we may take $G_s = \{1\}$ for almost all $s.$

This example also shows that if one weakens the conclusion in \cite[p.~51]{MP03} to assert that $M$ has finite monodromy, then the question becomes equivalent to the original $p$-curvature conjecture, since over an open $X$ in $\mathbb A^1_k$ and any object in  $MIC(X/k)$ with vanishing $p$-curvatures, we may take $G_s = \{ 1 \}$ for almost all $s.$  

Finally, we remark that in this whole discussion, we could have replaced $X \hookrightarrow \mathbb A^1_k$ by any smooth variety $X$ such that all vector bundles on $X_s$ are trivial. 
\end{para}

\end{document}